\documentclass[twoside,11pt]{amsart}
\usepackage{amsfonts}
\usepackage{amsmath}
\usepackage{amscd}
\usepackage{amssymb}
\usepackage{hyperref}
\usepackage{geometry}
\usepackage{graphicx}
\usepackage{inputenc}
\usepackage{enumerate}

\newtheorem{theorem}{Theorem}

\newtheorem{corollary}[theorem]{Corollary}

\newtheorem{definition}[theorem]{Definition}
\newtheorem{example}[theorem]{Example}

\newtheorem{lemma}[theorem]{Lemma}

\newtheorem{proposition}[theorem]{Proposition}
\newtheorem{remark}[theorem]{Remark}

\begin{document}
\title[Paley-Wiener Theorem]{A Paley-Wiener theorem for multi-anisotropic\\
ultradifferentiable functions}
\author{Chikh BOUZAR}
\address{Laboratory of Mathematical Analysis and Applications. Oran
University 1, Algeria.}
\email{ch.bouzar@gmail.com}
\date{}

\begin{abstract}
The present work presents a real Paley-Wiener Theorem within the space of
multi-anisotropic ultradifferentiable functions. We also show new properties
of this space.
\end{abstract}

\dedicatory{To the memory of Luiza Zanghirati}
\maketitle

\section{Introduction}

The multi-anisotropic Gevrey space $\mathcal{E}^{s,\Gamma }(\Omega ),$ whose
definition was given explicitly in the mathematical literature by Luisa
Zanghirati in the eighties of the last century in the paper \cite{Z}, has
been the subject of many works and applications. This space enters the class
of the so-called spaces of ultradifferentiable functions considered in
different fields of Mathematics as extensions or generalisations of the
classical space of real analytic functions $\mathcal{A}(\Omega ).$ A
well-known one of them that also interest us in this paper is the
Denjoy-Carleman space $\mathcal{E}^{M}\left( \Omega \right) ,$ see \cite%
{Carl}, \cite{Denj}, \cite{Horm}. The Beurling space $\mathcal{E}^{\omega
}\left( \Omega \right) $ in the sense of \cite{BMT}, see \cite{Beur} and
\cite{Bjor}, can be considered as an other facet of the Denjoy-Carleman
space, but not a duplicata. As more general than the space $\mathcal{E}%
^{s,\Gamma }(\Omega ),$ is the inhomogeneous Gevrey space $\mathcal{E}%
^{s,\lambda }(\Omega )$ of \cite{L-R}, see \cite{R}, defined in the spirit
of the space $\mathcal{C}^{L}(\Omega )$\ studied by L. H\"{o}rmander, see
\cite{Horm}.

Recall that the space $\mathcal{E}^{s,\Gamma }(\Omega )$\ has its defining
parameters the weight sequence $(k!^{s})_{k}$ and a regular Newton
polyhedron $\Gamma ,$\ and has its particular case the classical anisotropic
Gevrey space $\mathcal{E}^{s,\left( l_{1},...,l_{n}\right) }(\Omega )$ due
to M. Gevrey \cite{Gevr}. The space $\mathcal{E}^{M}\left( \Omega \right) $
is defined with the help of a weight sequence $M$\ and the space $\mathcal{E}%
^{\omega }\left( \Omega \right) $\ is defined with the help of a weight
function $\omega .$ While the space $\mathcal{E}^{s,\lambda }(\Omega )$\
depends on the\ weight sequence $(k!^{s})_{k}$ and an other type of weight
function $\lambda .$\ \ \ \

The study of the relationship between these cited spaces is a natural
question. The paper \cite{BMM} responded to this issue by finding the
relationship between the spaces $\mathcal{E}^{M}\left( \Omega \right) $\ and
$\mathcal{E}^{\omega }\left( \Omega \right) ,$ while the work \cite%
{CalvGomez} targeted the spaces $\mathcal{E}^{\omega }\left( \Omega \right) $
and $\mathcal{E}^{s,\lambda }(\Omega ).$ The paper \cite{Bouz} gives an
answer to this question for the spaces $\mathcal{E}^{M}\left( \Omega \right)
$\ and $\mathcal{E}^{s,\Gamma }(\Omega ),$ it also introduced a class of
ultradifferentiable functions, depending on a weight sequence $M$\ and a
regular Newton polyhedron $\Gamma ,$ unifying both spaces, denoted by $%
\mathcal{E}^{M,\Gamma }\left( \Omega \right) $ and called the space of
multi-anisotropic ultradifferentiable functions.

A Fourier analysis and even a micro-local analysis with respect to the $%
\mathcal{E}^{M}$-regularity\ is well-known, see e.g. H\"{o}rmander's book
\cite{Horm} and the references therein. What about the class $\mathcal{E}%
^{s,\Gamma }(\Omega )$ very few is done. Indeed, a first work on a
micro-local analysis with respect to the classical anisotropic Gevrey space
is done by L. Zanghirati in \cite{Z2}, then an attempt to go further to
study the $\mathcal{E}^{s,\Gamma }$-micro-regularity in the framework of
distributions\ is proposed by the paper \cite{Cor}\ following the approach
proposed within inhomogeneous Gevrey classes.

So, this paper has as its aim to present a first step for a Fourier analysis
related to the $\mathcal{E}^{M,\Gamma }$-regularity paving the way to a $%
\mathcal{E}^{M,\Gamma }$-micro-local analysis which will be done in a
forthcoming work. We prove a Paley-Wiener type Theorem within the space of
multi-anisotropic ultradifferentiable functions. We also show new properties
of the space $\mathcal{E}^{M,\Gamma }\left( \Omega \right) .$

The paper is organised as follows. The second section is devoted to recall
classical spaces of ultradifferentiable functions such the Denjoy-Carleman
space and the multi-anisotropic Gevrey space. The third section presents the
space of multi-anisotropic ultradifferentiable functions and some new
properties. In the last section a Paley-Wiener type Theorem for the space $%
\mathcal{E}^{M,\Gamma }\left( \Omega \right) $ is proved, as a consequence
of this result we compare, in the sense of inclusion, the spaces of
multi-anisotropic ultradifferentiable functions.

\section{Ultradifferentiable functions}

We keep the classical definitions and results of the theory of distributions
as in \cite{Horm}, as well as the notations. So $\mathcal{E}\left( \Omega
\right) $\ denotes the space of infinitely differentiable functions defined
in a non empty open set $\Omega $ of the Euclidean space $\mathbb{R}^{n}.$

\textbf{The space of multi-anisotropic Gevrey functions }$\mathcal{E}%
^{s,\Gamma }(\Omega ).$

Let $q=\left( q_{1},..,q_{n}\right) \in \mathbb{R}_{+}^{n}:=\left\{ x\in
\mathbb{R}^{n}:x_{j}>0,j=1,...,n\right\} $ and the multi-indice $\alpha
=\left( \alpha _{1},...,\alpha _{n}\right) \in \mathbb{Z}_{+}^{n},$ then $%
\alpha \cdot q:=\sum_{j=1}^{n}\alpha _{j}q_{j}$ $.$

\begin{definition}
The convex hull of $\left\{ 0\right\} \cup A,$ where $A$ is a finite subset
of $\overline{\mathbb{R}_{+}^{n}},$ is said a Newton polyhedron and denoted $%
\Gamma (A\mathbb{)}.$
\end{definition}

Given a Newton polyhedron $\Gamma $ there exists $\mathcal{A}\left( \Gamma
\right) $ a finite subset of $\mathbb{R}^{n}\diagdown \left\{ 0\right\} $
such that
\begin{equation*}
\Gamma \mathbb{=}\underset{q\in \mathcal{A}\left( \Gamma \right) }{\bigcap }%
\left\{ \alpha \in \overline{\mathbb{R}_{+}^{n}},\text{ }\alpha \cdot q\leq
1\right\}
\end{equation*}%
We say that $\Gamma $\ is regular if
\begin{equation*}
q_{j}>0,j=1,...,n,\forall q=\left( q_{1},...,q_{n}\right) \in \mathcal{A}%
\left( \Gamma \right)
\end{equation*}

\begin{remark}
The regularity of a Newton polyhedron $\Gamma $ means that for every point
of $\Gamma $ all its projections on the coordinate planes of various
dimensions belong to $\Gamma $ and it contains no faces parallel to the
coordinate hyperplanes and not belonging to them. As a consequence, $\Gamma $
has vertices on all the coordinates axes.
\end{remark}

The following elements are associated to every regular Newton polyhedron $%
\Gamma ,$
\begin{eqnarray*}
\mathcal{V}\left( \Gamma \right) &=&\left\{ v^{0}=0,\text{ }%
v^{1},...,v^{\sigma \left( \Gamma \right) }\right\} \text{ the finite set of
vertices of }\Gamma \\
k\left( \alpha ,\Gamma \right) &=&\inf \left\{ t>0,\frac{\alpha }{t}\in
\Gamma \right\} \text{ the jauge of }\Gamma \\
\mu _{j}\left( \Gamma \right) &=&\underset{q\in \mathcal{A}\left( \Gamma
\right) }{\max }q_{j}^{-1} \\
\mu \left( \Gamma \right) &=&\underset{1\leq \text{ }j\leq n}{\max }\mu
_{j}\left( \Gamma \right) \text{ the formal order of }\Gamma \\
\gamma (\Gamma ) &=&\underset{v\in \mathcal{V}\left( \Gamma \right) }{\max }%
\left\vert v\right\vert \text{ the maximal order of }\Gamma \\
l\left( \Gamma \right) &=&\left( \frac{1}{\mu _{1}\left( \Gamma \right) }%
,...,\frac{1}{\mu _{n}\left( \Gamma \right) }\right) \text{ the
complementary vector of }\Gamma
\end{eqnarray*}

The multi-anisotropic Gevrey space of the following definition was
introduced with a clear explicity in the mathematical literature by Luisa
Zanghirati in the paper \cite{Z} in order to study a Gevrey type regularity
of multi-quasi-elliptic linear partial differential operators by the method
of elliptic iterates, see also \cite{BCh2}, \cite{BD}, \cite{BD2} and \cite%
{Calv} for other applications. This space was well-known first as the
anisotropic Gevrey space in the study of parabolic and quasi-elliptic linear
partial differential operators, see \cite{Gevr} and \cite{Volev}.

\begin{definition}
Let $\Gamma $ be a regular Newton polyhedron and $s>0.$ The space of $u\in
\mathcal{E}(\Omega )$ satisfying for every compact $H$ of $\Omega $ there
exists $C>0$ such that for every $\alpha \in \mathbb{Z}_{+}^{n}$\ we have
\begin{equation}
\left\Vert \partial ^{\alpha }u\right\Vert _{L^{\infty }(H)}\leqslant
C^{\left\vert \alpha \right\vert +1}k(\alpha ,\Gamma )^{s\mu \left( \Gamma
\right) k(\alpha ,\Gamma )}  \label{4.1.1}
\end{equation}%
denoted $\mathcal{E}^{s,\Gamma }(\Omega ),$\ is called the multi-anisotropic
Gevrey space of order $s.$
\end{definition}

\begin{remark}
The value $k(\alpha ,\Gamma )^{s\mu \left( \Gamma \right) k(\alpha ,\Gamma
)} $\ can be replaced by the equivalent one $\left\vert \alpha \right\vert
^{s\mu \left( \Gamma \right) k(\alpha ,\Gamma )}.$
\end{remark}

\begin{example}
\label{anisotropic}Let $\left( l_{1},...,l_{n}\right) \in \mathbb{R}_{+}^{n}$
with $\min_{j}l_{j}=1.$\ The classical anisotropic Gevrey space is defined
and denoted as%
\begin{equation*}
\mathcal{E}^{s,\left( l_{1},...,l_{n}\right) }\left( \Omega \right) =\left\{
\begin{array}{c}
u\in \mathcal{E}\left( \Omega \right) :\forall H\text{ compact of }\Omega
,\exists C>0,\forall \alpha \in \mathbb{Z}_{+}^{n}, \\
\left\Vert \partial ^{\alpha }u\right\Vert _{L^{\infty }(H)}\leq
C^{\left\vert \alpha \right\vert +1}\alpha _{1}!^{sl_{1}}...\alpha
_{n}!^{sl_{n}}%
\end{array}%
\right\}
\end{equation*}%
We have $\mathcal{E}^{s,\left( l_{1},...,l_{n}\right) }(\Omega )=\mathcal{E}%
^{s,I_{q}}(\Omega ),$ where $I_{q}$ is the regular Newton polyhedron defined
by the simplex
\begin{equation*}
I_{q}:\mathbb{=}\left\{ \alpha \in \overline{\mathbb{R}_{+}^{n}}:\alpha
\cdot q\leq 1\right\} ,
\end{equation*}%
where $q=(\dfrac{l_{1}}{\max l_{j}},...,\dfrac{l_{n}}{\max l_{j}}).$ In such
case of the simplex $I_{q},$ we have%
\begin{eqnarray*}
\mathcal{V}\left( I_{q}\right) &=&\left\{ 0,\;\frac{\max l_{j}}{l_{j}}%
e_{j},\;j=1,..,n\right\} , \\
k(\alpha ,I_{q}) &=&\alpha \cdot q\  \\
\mu _{j}\left( I_{q}\right) &=&\frac{\max l_{j}}{l_{j}} \\
\mu \left( I_{q}\right) &=&\max l_{j} \\
l\left( I_{q}\right) &=&q
\end{eqnarray*}%
where $\left\{ e_{j},\;j=1,...,n\right\} \ $represents the canonical basis
of $\mathbb{R}^{n}.$
\end{example}

\begin{example}
The classical isotropic Gevrey space $\mathcal{E}^{s}\left( \Omega \right) $
coresponding to $l_{1}=...=l_{n}=1$\ is the space $\mathcal{E}^{s,I}(\Omega
),$ where $I:\mathbb{=}\left\{ \alpha \in \overline{\mathbb{R}_{+}^{n}}%
:\left\vert \alpha \right\vert \leq 1\right\} $\ is the unitary simplex. It
is well-know that $\mathcal{A}(\Omega )=\mathcal{E}^{1}\left( \Omega \right)
,$ i.e. $\mathcal{A}(\Omega )=\mathcal{E}^{1,I}(\Omega ).$ \
\end{example}

\begin{remark}
The study of the space $\mathcal{E}^{s,\Gamma }\left( \Omega \right) ,$
where $\Gamma $\ is a regular Newton polyhedron having vertices in $\mathbb{Q%
}_{+}^{n},$\ is reduced to the case of a regular Newton polyhedron having
vertices in $\mathbb{N}^{n}.$
\end{remark}

\textbf{The space of ultradifferentiable functions }$\mathcal{E}^{M}\left(
\Omega \right) .$

A sequence of positive real numbers $\left( M_{p}\right) _{p\in \mathbb{Z}%
_{+}}$ is said to satisfy the following condition of

logarithmic convexity, if $\forall p\in \mathbb{N},$%
\begin{equation}
M_{p}^{2}\leq M_{p-1}M_{p+1}  \tag{H1}
\end{equation}

stability under multiplication, if $\forall p,q\in \mathbb{Z}_{+},$%
\begin{equation}
M_{p}M_{q}\leq M_{0}M_{p+q}  \tag{H1'}
\end{equation}

stability under ultradifferential operators, if $\exists A>0,\exists
H>0,\forall p,q\in \mathbb{Z}_{+},$
\begin{equation}
M_{p+q}\leq AH^{p+q}M_{p}M_{q}  \tag{H2}
\end{equation}

stability under differential operators, if $\exists A>0,\exists H>0,\forall
p\in \mathbb{Z}_{+},$
\begin{equation}
M_{p+1}\leq AH^{p}M_{p}  \tag{H2'}
\end{equation}

non-quasianalyticity, if
\begin{equation}
\sum\limits_{p=1}^{\infty }\dfrac{M_{p-1}}{M_{p}}<+\infty  \tag{H3'}
\end{equation}

\begin{remark}
We always have $\left( Hj\right) \Rightarrow \left( Hj^{\prime }\right)
,j=1,2.$
\end{remark}

\begin{remark}
By the Denjoy-Carlemann Theorem, under $\left( H_{1}\right) ,$ the condition
$\left( H3^{\prime }\right) $ is equivalent to $\overset{\infty }{\underset{%
p=1}{\sum }}\dfrac{1}{M_{p}^{\frac{1}{p}}}<\infty .$\
\end{remark}

A $M$-ultradifferential operator is a differential operator of infinite
order $P=\sum\limits_{\alpha }a_{\alpha }D^{\alpha }$ satisfying for every $%
h>0$ there exists $c>0$ such that for every $\alpha \in \mathbb{Z}_{+}^{n}$\
we have
\begin{equation*}
\left\vert a_{\alpha }\right\vert M_{\left\vert \alpha \right\vert }\leq
ch^{\left\vert \alpha \right\vert }
\end{equation*}

\begin{definition}
The space of $M$-ultradifferentiable functions, denoted $\mathcal{E}%
^{M}\left( \Omega \right) ,$ is the set of $u\in \mathcal{E}\left( \Omega
\right) $ satisfying that for every compact $K$ of $\Omega ,\exists
C>0,\forall \alpha \in \mathbb{Z}_{+}^{n},$%
\begin{equation*}
\left\Vert \partial ^{\alpha }u\right\Vert _{L^{\infty }(K)}\leq
C^{\left\vert \alpha \right\vert +1}M_{\left\vert \alpha \right\vert }
\end{equation*}
\end{definition}

\begin{example}
If $\left( M_{p}\right) _{p\in \mathbb{Z}_{+}}=\left( p!^{s}\right) _{p\in
\mathbb{Z}_{+}},s>0,$ we obtain the classical isotropic Gevrey space $%
\mathcal{E}^{s}\left( \Omega \right) $ of order $s.$ The space of real
analytic functions $\mathcal{A}\left( \Omega \right) $ corresponds to the
space $\mathcal{E}^{1}\left( \Omega \right) .$
\end{example}

The basic properties of the space $\mathcal{E}^{M}\left( \Omega \right) $\
are summarized in the following proposition, see \cite{Kom} and \cite{Horm}
for more details on ultradifferentiable functions of Denjoy-Carleman type.

\begin{proposition}
If $\left( M_{p}\right) _{p\in \mathbb{Z}_{+}}$ satisfies $\left( H1^{\prime
}\right) $ the space $\mathcal{E}^{M}\left( \Omega \right) $\ is stable
under product. Moreover if $\left( M_{p}\right) _{p\in \mathbb{Z}_{+}}$
satisfies $\left( H2^{\prime }\right) $ then $\mathcal{E}^{M}\left( \Omega
\right) $ is stable by any linear differential operator of finite order with
coefficients from $\mathcal{E}^{M}\left( \Omega \right) ,$ and if $\left(
M_{p}\right) _{p\in \mathbb{Z}_{+}}$ satisfies $\left( H2\right) $ then any
ultradifferential operator of class $M$ operates as a sheaf homomorphism.

The space $\mathcal{D}^{M}\left( \Omega \right) :=\mathcal{E}^{M}\left(
\Omega \right) \cap \mathcal{D}\left( \Omega \right) $ is not trivial if and
only if the sequence $\left( M_{p}\right) _{p\in \mathbb{Z}_{+}}$ satisfies $%
\left( H3^{\prime }\right) .$
\end{proposition}

\section{Multi-anisotropic ultradifferentiable functions}

The study of the space of multi-anisotropic Gevrey vectors of systems of
linear differential operators reveals a characterization of
multi-anisotropic Gevrey spaces given by the following result, see \cite%
{Bouz} for the proof.

\begin{proposition}
\label{Char1}Let $\left( v_{j}\right) _{j=1}^{\sigma }\subset \mathbb{N}^{n}$
be the set of vertices of a regular Newton polyhedron $\Gamma $ and $s\geq 1$%
\ , then the following statements are equivalent for every $u\in \mathcal{E}%
\left( \Omega \right) ,$

i) $u\in \mathcal{E}^{s,\Gamma }\left( \Omega \right) .$

ii) $\forall H$ compact of $\Omega ,\exists C>0,\forall
k=(k_{1},...,k_{\sigma })\in \mathbb{Z}_{+}^{\sigma },$
\begin{equation}
\left\Vert \partial ^{k_{1}v_{1}}...\partial ^{k_{\sigma }v_{\sigma
}}u\right\Vert _{L^{\infty }\left( H\right) }\leq C^{\left\vert k\right\vert
+1}\left\vert k\right\vert ^{s\mu \left( \Gamma \right) \left\vert
k\right\vert }\quad .  \label{Char-1}
\end{equation}
\end{proposition}

As a consequence we obtain a characterisation of the classical anisotropic
Gevrey spaces.

\begin{corollary}
Let $\left( l_{1},..,l_{n}\right) \in \mathbb{Q}_{+}^{n}$ \ with $%
\min_{j}l_{j}=1,$ then the anisotropic Gevrey space $\mathcal{E}^{s,\left(
l_{1},..,l_{n}\right) }\left( \Omega \right) $ coincides with the space of
functions $u\in \mathcal{E}\left( \Omega \right) $\ such that $\forall H$
compact of $\Omega ,\exists C>0,\forall k\in \mathbb{Z}_{+}^{n},$ we have
\begin{equation*}
\left\Vert \partial _{1}^{k_{1}m_{1}}...\partial
_{n}^{k_{n}m_{n}}u\right\Vert _{L^{\infty }\left( H\right) }\leq
C^{\left\vert k\right\vert +1}\left\vert k\right\vert ^{sm\left\vert
k\right\vert },
\end{equation*}%
where $m=\max l_{j}$\ and such that $m_{j}:=\dfrac{m}{l_{j}}\in \mathbb{N}%
,j=1,...,n.$\
\end{corollary}

\begin{example}
A function $u\in \mathcal{E}^{s,(2,1)}(\Omega )$\ if and only if $u\in
\mathcal{E}\left( \Omega \right) $\ and for every compact $H$ of $\Omega $\
there exists $C>0$ such that%
\begin{equation*}
\left\Vert \partial _{t}^{k_{1}}\partial _{x}^{2k_{2}}u\right\Vert
_{L^{\infty }\left( H\right) }\leq
C^{k_{1}+k_{2}+1}(k_{1}+k_{2})!^{2s},\forall k=(k_{1},k_{2})\in \mathbb{Z}%
_{+}^{2}.
\end{equation*}%
These estimates allow to show the anisotropic Gevrey regularity of the
solutions of the heat equation. \
\end{example}

Let $\left( v_{j}\right) _{j=1}^{\sigma }\subset \mathbb{N}^{n}$ be the set
of vertices of a regular Newton polyhedron $\Gamma ,$ the differential
operator%
\begin{equation*}
\partial _{\Gamma }:=\partial ^{v_{1}}\cdots \partial ^{v_{\sigma }}
\end{equation*}%
is said the $\Gamma $-derivation. The mixed derivative of order $%
k=(k_{1},...,k_{\sigma })\in \mathbb{Z}_{+}^{\sigma }$ is defined as%
\begin{equation*}
\partial _{\Gamma }^{k}:=\partial ^{k_{1}v_{1}}\cdots \partial ^{k_{\sigma
}v_{\sigma }}
\end{equation*}

All what have been said legitimizes to introduce the space of
multi-anisotropic ultradifferentiable functions.

\begin{definition}
Let $M=\left( M_{j}\right) _{j=0}^{\infty }$ be a sequence\ of real positive
numbers and $\Gamma $ a regular Newton polyhedron. The space of
multi-anisotropic ultradifferentiable functions, denoted $\mathcal{E}%
^{M,\Gamma }\left( \Omega \right) ,$ is defined as the space of functions $%
u\in \mathcal{E}\left( \Omega \right) $\ such that $\forall H$ compact of $%
\Omega ,\exists C>0,\forall k\in \mathbb{Z}_{+}^{\sigma },$
\begin{equation}
\left\Vert \partial _{\Gamma }^{k}u\right\Vert _{L^{\infty }\left( H\right)
}\leq C^{\left\vert k\right\vert +1}M_{\left\vert k\right\vert }  \label{Def}
\end{equation}
\end{definition}

The spaces $\mathcal{E}^{M,\Gamma }\left( \Omega \right) $ gives as examples
the cited classical spaces of ultradifferentiable functions.

\begin{example}
The classical Denjoy-Carleman space $\mathcal{E}^{M}\left( \Omega \right) $
corresponds to the space $\mathcal{E}^{M,\Gamma }\left( \Omega \right) $\
with $\Gamma $ the identity simplex $I.$ In particular, the classical Gevrey
space $\mathcal{E}^{s}(\Omega )$ has $M$ equals the Gevrey sequence $\left(
p^{sp}\right) _{p}$
\end{example}

\begin{example}
The classical anisotropic Gevrey space $\mathcal{E}^{s,\left(
l_{1},..,l_{n}\right) }(\Omega )$ corresponds to $\mathcal{E}%
^{M,I_{q}}\left( \Omega \right) ,$ where $M=$ $\left( p^{s\mu p}\right) _{p}$
and $\mu =\max l_{j}.$ Recalling that $I_{q}\mathbb{=}\left\{ \alpha \in
\overline{\mathbb{R}_{+}^{n}}:\alpha \cdot q\leq 1\right\} $ and $q=(\dfrac{%
l_{1}}{\mu },...,\dfrac{l_{n}}{\mu }).$
\end{example}

\begin{example}
The multi-anisotropic Gevrey space $\mathcal{E}^{s,\Gamma }(\Omega )$
corresponds to $\mathcal{E}^{M,\Gamma }\left( \Omega \right) $ with $M$
equals the Gevrey sequence $\left( p^{s\mu (\Gamma )p}\right) _{p},$ see
Proposition (\ref{Char1}).
\end{example}

We present a concrete example coming from linear partial differential
operators, for more details see \cite{BD}\ and \cite{BD2}.

\begin{example}
The composition of the Cauchy-Riemann operator with the heat operator,
denoted here $P(D),$ being the most simple multi-quasielliptic operator, has
its symbol the polynomial%
\begin{equation*}
P\left( \xi ,\eta \right) =\xi ^{2}-i\xi \eta ^{2}+i\xi \eta +\eta ^{3},
\end{equation*}%
which generates a regular Newton polyedron $\Gamma \left( P\right) ,$ see
\cite{GV}, with associated elements%
\begin{equation*}
\mathcal{V}\left( P\right) =\left\{ \left( 0,0\right) ,\left( 2,0\right)
,\left( 1,2\right) ,\left( 0,3\right) \right\}
\end{equation*}%
\begin{equation*}
\mathcal{A}\left( P\right) =\left\{ \left( \frac{1}{2},\frac{1}{4}\right)
,\left( \frac{1}{3},\frac{1}{3}\right) \right\}
\end{equation*}%
\begin{equation*}
\mu \left( P\right) =4,
\end{equation*}%
\begin{equation*}
k\left( \alpha ,P\right) =\left\{
\begin{array}{c}
\frac{1}{2}\alpha _{1}+\frac{1}{4}\alpha _{2},if\text{\ }\alpha _{2}\leq
2\alpha _{1} \\
\frac{1}{3}\alpha _{1}+\frac{1}{3}\alpha _{2},if\text{ }\alpha _{2}\geq
2\alpha _{1}%
\end{array}%
\right.
\end{equation*}%
and we have%
\begin{equation*}
G^{s,\text{ }\Gamma \left( P\right) }\left( \Omega \right) =\left\{
\begin{array}{c}
u\in \mathcal{C}^{\infty }\left( \Omega \right) ,\text{ }\forall K\subset
\subset \Omega ,\text{ }\exists C>0,\text{ }\forall \left( \alpha
_{1},\alpha _{2}\right) \in \mathbb{Z}_{+}\times \mathbb{Z}_{+}. \\
\underset{x\in K}{\sup }\left\vert D^{\alpha }u\left( x\right) \right\vert
\leq C^{\alpha _{1}+\alpha _{2}+1}\left( \frac{1}{2}\alpha _{1}+\frac{1}{4}%
\alpha _{2}\right) ^{4s(\frac{1}{2}\alpha _{1}+\frac{1}{4}\alpha _{2})},%
\text{ if }\alpha _{2}\leq 2\alpha _{1}. \\
\underset{x\in K}{\sup }\left\vert D^{\alpha }u\left( x\right) \right\vert
\leq C^{\alpha _{1}+\alpha _{2}+1}\left( \frac{1}{3}\alpha _{1}+\frac{1}{3}%
\alpha _{2}\right) ^{4s(\frac{1}{3}\alpha _{1}+\frac{1}{3}\alpha _{2})},%
\text{ if }\alpha _{2}\geq 2\alpha _{1}.%
\end{array}%
\right\}
\end{equation*}%
It is known that the operator $P\left( D\right) $ is hypoelliptic and if $%
u\in \mathcal{D}^{\prime }\left( \Omega \right) $ satisfies the equation $%
P\left( D\right) u=0$ then we have $u\in G^{s,\Gamma \left( P\right) }\left(
\Omega \right) ,$ but in the language of the space of multi-anisotropic
ultradifferentiable functions, we have that%
\begin{equation*}
G^{s,\text{ }\Gamma \left( P\right) }\left( \Omega \right) =\left\{
\begin{array}{c}
u\in \mathcal{C}^{\infty }\left( \Omega \right) ,\text{ }\forall K\subset
\Omega ,\text{ }\exists C>0,\text{ }\forall \alpha =\left( \alpha
_{1},\alpha _{2},\alpha _{3}\right) \in \mathbb{Z}_{+}\times \mathbb{Z}%
_{+}\times \mathbb{Z}_{+}, \\
\underset{x\in K}{\sup }\left\vert \partial _{x_{1}}^{2\alpha _{1}+\alpha
_{2}}\partial _{x_{2}}^{2\alpha _{2}+3\alpha _{3}}u\left( x\right)
\right\vert \leq C^{\left\vert \alpha \right\vert +1}\left\vert \alpha
\right\vert ^{4s\left\vert \alpha \right\vert }.%
\end{array}%
\right\}
\end{equation*}%
These are the more precise estimates of ultradifferentiability type for the
solutions of the equation $P\left( D\right) u=0.$
\end{example}

\begin{remark}
\label{H1H2}Conditions $\left( H1^{\prime }\right) ,\left( H2\right) $
implies $\exists A>0,\exists C>,\exists H>0,\forall (p_{1},...,p_{l})\in
\mathbb{Z}_{+}^{l},$
\begin{equation*}
A^{l-1}M_{p_{1}}\times ...\times M_{p_{l}}\leq M_{p_{1}+...+p_{l}}\leq
C^{l-1}H^{p_{1}+...+p_{l}}M_{p_{1}}\times ...\times M_{p_{l}}
\end{equation*}%
So we will write $M_{p_{1}+...+p_{l}}\approx M_{p_{1}}\times ...\times
M_{p_{l}}.$
\end{remark}

\begin{proposition}
Given $M=\left( M_{p}\right) _{p=0}^{\infty }$ a sequence\ of real positive
numbers satisfying $\left( H1^{\prime }\right) $ and $\left( H2\right) ,$
and $\Gamma $ a regular Newton polyhedron, then there exists a sequence of
positive numbers $N=\left( N_{p}\right) _{p=0}^{\infty }$\ such that $%
\mathcal{E}^{N}\left( \Omega \right) \subset \mathcal{E}^{M,\Gamma }\left(
\Omega \right) .$
\end{proposition}

\begin{proof}
Let $\mathcal{V}\left( \Gamma \right) =\left\{ v^{1},...,v^{\sigma }\right\}
$ be the finite set of vertices of $\Gamma $ and $%
v^{j}=(v_{1}^{j},...,v_{n}^{j})\in \mathbb{R}_{+}^{n}.$\ Let $u\in \mathcal{E%
}\left( \Omega \right) ,$\ then for every $k=(k_{1},...,k_{\sigma })\in
\mathbb{Z}_{+}^{\sigma },$ we have%
\begin{equation*}
\partial _{\Gamma }^{k}u=\partial ^{\alpha }u,
\end{equation*}%
where $\alpha =(\alpha _{1},...,\alpha _{n})\in \mathbb{Z}_{+}^{n}$\ is such
that%
\begin{equation*}
\alpha _{j}=\sum\limits_{l=1}^{\sigma }k_{l}v_{j}^{l},j=1,...,n.
\end{equation*}%
So, if $u\in \mathcal{E}^{N}\left( \Omega \right) ,$ then for every $K$
compact of $\Omega $ there exists $C>0$\ such that%
\begin{equation*}
\left\Vert \partial _{\Gamma }^{k}u\right\Vert _{L^{\infty }\left( K\right)
}=\left\Vert \partial ^{\alpha }u\right\Vert _{L^{\infty }\left( K\right)
}\leq C^{\left\vert \alpha \right\vert +1}N_{\left\vert \alpha \right\vert },
\end{equation*}%
where%
\begin{equation*}
\left\vert \alpha \right\vert =\sum\limits_{l=1}^{\sigma }k_{l}\left\vert
v^{l}\right\vert
\end{equation*}%
Due to the properties $\left( H1^{\prime }\right) $ and $\left( H2\right) ,$%
\ see Remark (\ref{H1H2}), we have
\begin{equation*}
N_{\left\vert \alpha \right\vert }\approx \prod\limits_{l=1}^{\sigma
}N_{k_{l}\left\vert v^{l}\right\vert }\approx \prod\limits_{l=1}^{\sigma
}(N_{k_{l}})^{\left\vert v^{l}\right\vert }\preceq
\prod\limits_{l=1}^{\sigma }(N_{k_{l}})^{\gamma }\preceq (N_{\left\vert
k\right\vert })^{\gamma },
\end{equation*}%
where $\gamma =\max_{l}\left\vert v^{l}\right\vert >0$ is the maximal order
of $\Gamma .$\ Defining $N=\left( M_{p}^{\frac{1}{\gamma }}\right)
_{p=0}^{\infty }$\ we obtain the inclusion $\mathcal{E}^{N}\left( \Omega
\right) \subset \mathcal{E}^{M,\Gamma }\left( \Omega \right) .$\
\end{proof}

\begin{remark}
We have $\mathcal{E}^{M^{\frac{1}{\gamma }}}\left( \Omega \right) \subset
\mathcal{E}^{M,\Gamma }\left( \Omega \right) ,$ where $\gamma $ is the
maximal order of $\Gamma ,$\ under the condition that the sequence $M$\
satisfies the conditions $\left( H1^{\prime }\right) $ and $\left( H2\right)
.$
\end{remark}

\begin{corollary}
Let the sequence $M=\left( M_{p}\right) _{p=0}^{\infty }$ satisfying $\left(
H1^{\prime }\right) ,\left( H2\right) ,$ and the sequence $\left( M_{p}^{%
\frac{1}{\gamma }}\right) _{p=0}^{\infty }$ satisfies $\left( H3^{\prime
}\right) ,$ then the space $\mathcal{E}^{M,\Gamma }\left( \Omega \right) $
is non-quasianalytic.\ \ \
\end{corollary}

\begin{definition}
A linear $\Gamma $-differential operator $P$ of order $m\in \mathbb{N}$ with
constant coefficients $(a_{k})_{k}\in \mathbb{C}$ is defined as the operator%
\begin{equation*}
P=\sum\limits_{\left\vert k\right\vert \leq m}a_{k}\partial _{\Gamma }^{k}
\end{equation*}
\end{definition}

It is easy to prove the following result.

\begin{proposition}
The space $\mathcal{E}^{M,\Gamma }\left( \Omega \right) $ is a vector space
stable under action of linear $\Gamma $-differential operators.
\end{proposition}

We have a more general result, see \cite{Bouz} for the proof.

\begin{definition}
A $\Gamma $-ultradifferential operator of type $M$ is a differential
operator of infinte order $\sum\limits_{\gamma }a_{\gamma }\partial _{\Gamma
}^{\gamma }$ satisfying for every $h>0$ there exists $c>0$ such that $%
\forall \gamma \in \mathbb{Z}_{+}^{\sigma }$ we have%
\begin{equation*}
\left\vert a_{\gamma }\right\vert M_{\left\vert \gamma \right\vert }\leq
ch^{\left\vert \gamma \right\vert }
\end{equation*}
\end{definition}

\begin{proposition}
The space $\mathcal{E}^{M,\Gamma }\left( \Omega \right) $ is stable under
the action of $\ \Gamma $-ultradifferential operator of type $M.$
\end{proposition}

The following diagram summarizes the development of the spaces of
ultradifferentiable functions,
\begin{equation*}
\begin{tabular}{lllllll}
&  &  &  & $\mathcal{E}^{M,I}\left( \Omega \right) $ &  &  \\
&  &  & $\nearrow $ &  & $\searrow $ &  \\
$\mathcal{A}(\Omega )$ & $\rightarrow $ & $\mathcal{E}^{s,I}(\Omega
)\rightarrow \mathcal{E}^{s,\left( l_{1},..,l_{n}\right) }\left( \Omega
\right) $ &  &  &  & $\mathcal{E}^{M,\Gamma }\left( \Omega \right) $ \\
&  &  & $\searrow $ &  & $\nearrow $ &  \\
&  &  &  & $\mathcal{E}^{s,\Gamma }(\Omega )$ &  &
\end{tabular}%
\end{equation*}

\section{Real Paley-Wiener Theorem}

Recall that with a regular Newton polyhedron $\Gamma $ is associated its
elements $\mathcal{A}\left( \Gamma \right) ,\mu \left( \Gamma \right) $ and $%
\mathcal{V}\left( \Gamma \right) =\left( v_{j}\right) _{j=1}^{\sigma }$ the
set of its vertices.

\begin{definition}
The weight function associated with $\Gamma $ is defined and denoted as the
positive function%
\begin{equation*}
\left\vert \xi \right\vert _{\Gamma }:=\left( \sum\limits_{v\in \mathcal{V}%
\left( \Gamma \right) }\left\vert \xi ^{v}\right\vert ^{2}\right) ^{\frac{1}{%
2\mu }},\xi \in \mathbb{R}^{n}.
\end{equation*}
\end{definition}

\begin{example}
In the anisotropic case, as particular case, and $l=\left(
l_{1},...,l_{n}\right) \in \mathbb{R}_{+}^{n},$ we have%
\begin{equation*}
\left\vert \xi \right\vert _{l}:=\left( \sum\limits_{j=1}^{n}\left\vert \xi
_{j}\right\vert ^{\frac{2}{l_{j}}}\right) ^{\frac{1}{2}},\xi \in \mathbb{R}%
^{n}.
\end{equation*}
\end{example}

First, we need the following result.

\begin{lemma}
Let $\Gamma $ and $\Gamma ^{\prime }$\ be two regular Newton polyhedrons,
then $\Gamma ^{\prime }\subset \Gamma $\ if and only if there exists $c>0$\
such that%
\begin{equation*}
\left\vert \xi \right\vert _{\Gamma ^{\prime }}^{\mu ^{\prime }}\leq
c\left\vert \xi \right\vert _{\Gamma }^{\mu },\forall \xi \in \mathbb{R}^{n}.
\end{equation*}
\end{lemma}

\begin{proof}
It is clear to see the following%
\begin{equation*}
\forall \alpha \in \mathbb{R}_{+}^{n}\diagdown \left\{ 0\right\} ,\frac{%
\alpha }{k(\alpha ,\Gamma )}\in \Gamma ,
\end{equation*}%
in particular $\alpha \in \Gamma \Leftrightarrow k(\alpha ,\Gamma )\leq 1,$\
giving that%
\begin{equation*}
\forall \alpha \in \mathbb{R}_{+}^{n},\exists c>0,\forall \xi \in \mathbb{R}%
^{n},\left\vert \xi ^{\alpha }\right\vert \leq c\left\vert \xi \right\vert
_{\Gamma }^{\mu k(\alpha ,\Gamma )}
\end{equation*}%
Consequently, if $\Gamma ^{\prime }\subset \Gamma $\ there exists $c>0$\
such that $\left\vert \xi \right\vert _{\Gamma ^{\prime }}^{\mu ^{\prime
}}\leq c\left\vert \xi \right\vert _{\Gamma }^{\mu },\forall \xi \in \mathbb{%
R}^{n}.$

Conversely, suppose $\exists c>0,\forall \xi \in \mathbb{R}^{n},\left\vert
\xi \right\vert _{\Gamma ^{\prime }}^{\mu ^{\prime }}\leq c\left\vert \xi
\right\vert _{\Gamma }^{\mu },$ and there exists $\alpha _{0}\in \Gamma
^{\prime }$\ such that $\alpha _{0}\notin \Gamma .$\ Using the separation
theorem of convexes due to the regularity of the Newton polyhedron $\Gamma ,$
there exists $q\in \mathbb{R}_{+}^{n}$\ such that $\alpha \cdot q<0,\forall
\alpha \in \Gamma ,$\ and $\alpha _{0}\cdot q>0.$ Taking a vector $\eta \in
\mathbb{R}_{+}^{n},\eta _{1}\times ...\times \eta _{n}\neq 0,$\ as we have%
\begin{equation*}
\forall t>0,\left\vert (t^{q}\eta )^{\alpha _{0}}\right\vert \leq
c\left\vert t^{q}\eta \right\vert _{\Gamma }^{\mu k(\alpha _{0},\Gamma )},
\end{equation*}%
i.e. $\forall t>0,$%
\begin{equation*}
t^{q\cdot \alpha _{0}}\left\vert \eta \right\vert \leq c\left(
\sum\limits_{v\in \mathcal{V}\left( \Gamma \right) }t^{2q\cdot v}\left\vert
\eta ^{v}\right\vert ^{2}\right) ^{\frac{k(\alpha _{0},\Gamma )}{2}}
\end{equation*}%
and let $t\rightarrow +\infty ,$\ we obtain a contradiction.
\end{proof}

We recall an important element associated with a sequence $(M_{p})_{p\in
\mathbb{Z}_{+}}.$

\begin{definition}
The associated function of a sequence $(M_{p})_{p\in \mathbb{Z}_{+}}$ is
defined as the function
\begin{equation*}
M(t):=\sup_{p}\log \frac{t^{p}}{M_{p}},t>0.
\end{equation*}
\end{definition}

\begin{example}
The Gevrey case defined by the sequence $(p!)^{s}$ has its associated
function equivalent to the function $t^{\frac{1}{s}},t>0.$
\end{example}

We present the following Paley-Wiener type Theorem for multi-anisotropic
ultradifferentiable functions.

\begin{theorem}
Let $u\in \mathcal{E}^{\prime },$ then the following assertions are
equivalent

(i) $u\in \mathcal{E}^{M,\Gamma }(\Omega ).$

(ii) there exist $C>0$\ and $a>0$ such that%
\begin{equation*}
\left\vert \widehat{u}\left( \zeta \right) \right\vert \leq Ce^{\left(
H_{K}(\Im\zeta )-M(a\left\vert \zeta \right\vert _{\Gamma })\right) },\zeta
\in \mathbb{C}^{n},
\end{equation*}%
where $H_{K}(\Im\zeta ):=\sup_{x\in K}\left\vert \Im\zeta \cdot x\right\vert
$ and $K=suppu.$
\end{theorem}

\begin{proof}
Let $u\in \mathcal{E}^{L,\Gamma }\left( \Omega \right) \cap \mathcal{E}%
^{\prime },$\ then $\widehat{u}$\ beeing an entire function, this justifies
the equality%
\begin{equation*}
\zeta ^{k_{1}v_{1}}...\zeta ^{k_{\sigma }v_{\sigma }}\widehat{u}\left( \zeta
\right) =(2\pi )^{-n}\int\limits_{\mathbb{R}^{n}}e^{-ix\cdot \zeta }\partial
_{\Gamma }^{k}u\left( x\right) dx,\forall k=(k_{1},...,k_{\sigma })\in
\mathbb{Z}_{+}^{\sigma },
\end{equation*}%
where $\left( v_{j}\right) _{j=1}^{\sigma }$ are the vertices of the
polyhedron $\Gamma ,$\ so we have the estimates%
\begin{equation*}
\left\vert \zeta ^{k_{1}v_{1}}...\zeta ^{k_{\sigma }v_{\sigma }}\widehat{u}%
\left( \zeta \right) \right\vert \leq \int\limits_{\mathbb{R}^{n}}\left\vert
e^{-ix\cdot \zeta }\partial _{\Gamma }^{k}u\left( x\right) \right\vert dx,
\end{equation*}%
i.e.%
\begin{equation*}
\left\vert \zeta ^{v_{1}}\right\vert ^{k_{1}}...\left\vert \zeta ^{v_{\sigma
}}\right\vert ^{k_{\sigma }}\left\vert \widehat{u}\left( \zeta \right)
\right\vert \leq \int\limits_{K}e^{\left\vert \Im\zeta \cdot x\right\vert
}\left\vert \partial _{\Gamma }^{k}u\left( x\right) \right\vert dx,
\end{equation*}%
from which, varying the values of the multi-indices $k=(k_{1},...,k_{\sigma
})\in \mathbb{Z}_{+}^{\sigma },$ we obtain%
\begin{equation*}
\left( \left\vert \xi \right\vert _{\Gamma }^{\mu }\right) ^{\left\vert
k\right\vert }\left\vert \widehat{u}\left( \zeta \right) \right\vert \leq
Ae^{H_{K}(\Im\zeta )}\left( CL_{\left\vert k\right\vert }\right)
^{\left\vert k\right\vert },\forall \zeta \in \mathbb{C}^{n},
\end{equation*}%
with a constant $A>0$ \ independent of $k\in \mathbb{Z}_{+}^{\sigma }.$
Consequently, the desired estimate (ii) is obtained.

Conversely, let (ii) be satisfied then there exist $C>0$\ and $a>0$ such
that $\forall \xi \in \mathbb{R}^{n},$%
\begin{equation*}
\left\vert \widehat{u}\left( \xi \right) \right\vert \leq Ce^{-M(a\left\vert
\xi \right\vert _{\Gamma }^{\mu })},
\end{equation*}%
which is equivalent to%
\begin{equation*}
\sup_{p}\frac{a^{p}\left\vert \xi \right\vert _{\Gamma }^{p\mu }}{M_{p}}%
\left\vert \widehat{u}\left( \xi \right) \right\vert \leq C
\end{equation*}%
The inverse Fourier transform gives%
\begin{equation*}
\partial _{\Gamma }^{k}u\left( x\right) =\int\limits_{\mathbb{R}%
^{n}}e^{ix\cdot \xi }\xi ^{k_{1}v_{1}}...\xi ^{k_{\sigma }v_{\sigma }}%
\widehat{u}\left( \xi \right) d\xi ,\forall k=(k_{1},...,k_{\sigma })\in
\mathbb{Z}_{+}^{\sigma },\forall x\in \Omega .
\end{equation*}%
Consequently, we have the estimates, $\forall k=(k_{1},...,k_{\sigma })\in
\mathbb{Z}_{+}^{\sigma },\forall x\in \Omega ,$%
\begin{eqnarray*}
\left\vert \partial _{\Gamma }^{k}u\left( x\right) \right\vert &\leq
&\int\limits_{\mathbb{R}^{n}}\frac{M_{\left\vert k\right\vert }\left\vert
\xi ^{k_{1}v_{1}}...\xi ^{k_{\sigma }v_{\sigma }}\right\vert }{a^{\left\vert
k\right\vert }\left\vert \xi \right\vert _{\Gamma }^{\left\vert k\right\vert
\mu }}\frac{a^{\left\vert k\right\vert }\left\vert \xi \right\vert _{\Gamma
}^{\left\vert k\right\vert \mu }}{M_{\left\vert k\right\vert }}\widehat{u}%
\left( \xi \right) d\xi \\
&\leq &C\frac{1}{a^{\left\vert k\right\vert }}M_{\left\vert k\right\vert
}\int\limits_{\mathbb{R}^{n}}\frac{\left\vert \xi ^{k_{1}v_{1}}...\xi
^{k_{\sigma }v_{\sigma }}\right\vert }{\left\vert \xi \right\vert _{\Gamma
}^{\left\vert k\right\vert \mu }}d\xi \leq CC^{\prime }\frac{1}{%
a^{\left\vert k\right\vert }}M_{\left\vert k\right\vert },
\end{eqnarray*}%
which imply that $u\in \mathcal{E}^{M,\Gamma }(\Omega ).$
\end{proof}

\begin{corollary}
Let $u\in \mathcal{E}^{\prime },$ then the following assertions are
equivalent

(1) $u\in \mathcal{E}^{M,\Gamma }(\Omega ).$

(2) there exist $C>0$\ and $a>0$ such that $\forall \xi \in \mathbb{R}^{n}$
we have%
\begin{equation*}
\left\vert \widehat{u}\left( \xi \right) \right\vert \leq Ce^{-M(a\left\vert
\xi \right\vert _{\Gamma }^{\mu })}
\end{equation*}
\end{corollary}

We recall a well-known result.\ If $M$ and $N$\ are two sequences of
positive numbers, $N\prec M$ means there exist $A>0,C>0$\ such that $%
N_{p}\leq CA^{p}M_{p},\forall p\in \mathbb{Z}_{+}.$

\begin{lemma}
The sequences $M$ and $N$\ satisfy $N\prec M$ if and only if there exist $%
a>0,C>0$ such that%
\begin{equation*}
\ \ e^{M(t)}\leq Ce^{N(at)},t>0.
\end{equation*}
\end{lemma}

We know give a comparison between the spaces $\mathcal{E}^{M,\Gamma }(\Omega
).$ Under the conditions of non-quasianalyticity on the sequence $M, $ we
define\ $\mathcal{D}^{M,\Gamma }(\Omega ):=\mathcal{E}^{M,\Gamma }(\Omega
)\cap \mathcal{E}^{\prime }.$

\begin{proposition}
Let given sequences $M,N$ and regular Newton polyhedrons $\Gamma ,\Lambda ,$%
\ then we have%
\begin{equation*}
N\prec M\text{ and }\Gamma \subset \Lambda \Rightarrow \mathcal{D}%
^{N,\Lambda }(\Omega )\subset \mathcal{D}^{M,\Gamma }(\Omega )
\end{equation*}
\end{proposition}

\begin{proof}
The conditions $N\prec M$ and $\Gamma \subset \Lambda $\ are\ equivalent,
respectively, to the conditions $\ e^{-N(at)}\leq Ce^{-M(t)},t>0,$\ and $%
\left\vert \xi \right\vert _{\Gamma }^{\mu }\leq c\left\vert \xi \right\vert
_{\Lambda }^{\lambda },\xi \in \mathbb{R}^{n}.$\ As the functions $M$\ and $%
M^{\prime }$\ are increasing, then we obtain%
\begin{eqnarray*}
e^{-N(a\left\vert \xi \right\vert _{\Lambda }^{\lambda })} &\leq
&Ce^{-M(\left\vert \xi \right\vert _{\Lambda }^{\lambda })}, \\
-M(c\left\vert \xi \right\vert _{\Lambda }^{\lambda }) &\leq &-M(\left\vert
\xi \right\vert _{\Gamma }^{\mu }), \\
e^{-M(c\left\vert \xi \right\vert _{\Lambda }^{\lambda })} &\leq
&e^{-M(\left\vert \xi \right\vert _{\Gamma }^{\mu })},
\end{eqnarray*}%
which give that there exist constants $a>0,C>0$\ such that%
\begin{equation*}
e^{-N(a\left\vert \xi \right\vert _{\Lambda }^{\lambda })}\leq
Ce^{-M(\left\vert \xi \right\vert _{\Gamma }^{\mu })},\xi \in \mathbb{R}^{n}.
\end{equation*}%
Applying the last Corollary we obtain the desired inclusion.
\end{proof}

\begin{remark}
The general inclusion $\mathcal{E}^{N,\Lambda }(\Omega )\subset \mathcal{E}%
^{M,\Gamma }(\Omega )$\ needs an other proof.
\end{remark}

\textbf{Conflict of interest statement.}\newline
I formally declare that I did not have any financial, professional, or
relationships that could be perceived as prejudicing this research paper.

\end{document}